\documentclass{amsart}
\usepackage{amssymb,latexsym,amsmath,amscd,epsfig,amsthm,amsfonts}

\usepackage{url}
\usepackage{enumerate}

\newtheorem*{facta}{Theorem A}
\newtheorem*{factb}{Theorem B}
\newtheorem*{factc}{Theorem C}
\newtheorem*{factal}{Lemma A}

\newtheorem{thrm}{Theorem}[section]
\newtheorem{lem}[thrm]{Lemma}
\newtheorem{prop}[thrm]{Proposition}
\newtheorem{cor}[thrm]{Corollary}
\theoremstyle{definition}
\newtheorem{definition}[thrm]{Definition}
\newtheorem{remark}[thrm]{Remark}

\numberwithin{equation}{section}

\newcommand{\D}{{\mathbf D}}

\newcommand{\B}{{\mathbf B}}
\newcommand{\N}{{\mathbf N}}
\newcommand{\R}{{\mathbf R}}

\newcommand{\beq}{\begin{equation}}
\newcommand{\eeq}{\end{equation}}
\newcommand{\beqs}{\begin{equation*}}
\newcommand{\eeqs}{\end{equation*}}
\newcommand{\ba}{\begin{align}}
\newcommand{\ea}{\end{align}}
\newcommand{\bas}{\begin{align*}}
\newcommand{\eas}{\end{align*}}

\newcommand{\bg}{\begin{gather}}
\newcommand{\eg}{\end{gather}}
\newcommand{\bgs}{\begin{gather*}}
\newcommand{\egs}{\end{gather*}}

\newcommand{\E}{\mathcal {E}}
\newcommand{\pr}{\mathcal {P}}

\newcommand{\z}{\zeta}
\newcommand{\x}{\xi}

\newcommand{\ta}{\theta}
\newcommand{\fa}{\mathbf{a}}

\newcommand{\infinity}{\infty}
\newcommand {\M} {\mathfrak{M}}
\newcommand {\hv} {h^\infty_v(\D)}
\newcommand {\hvb} {h^\infty_v(\B)}

\newcommand{\bl}{\mathcal B}

\newcommand{\p}{\phi}

\newcommand{\fraction}{\frac}

\begin{document}

\author{Kjersti Solberg Eikrem} 

\address{Department of Mathematical Sciences, Norwegian University of Science and
Technology, NO-7491 Trondheim, Norway}
\email{kjerstei@math.ntnu.no}


\title
{Random harmonic functions in growth spaces and Bloch-type spaces} 

\begin{abstract}
Let $h^\infty_v(\mathbf D)$ and $h^\infty_v(\mathbf B)$ be the spaces of harmonic functions in the unit disk and  multi-dimensional unit ball
 which admit a two-sided radial majorant $v(r)$. 
  We consider functions $v $ that fulfill a doubling condition. In the two-dimensional case let $u (re^{i\ta},\xi) = \sum_{j=0}^\infty (a_{j0} \xi_{j0} r^j \cos j\theta +a_{j1} \xi_{j1} r^j \sin j\theta)$  where $\xi =\{\xi_{ji}\}
 $  is a  sequence of random subnormal variables and $a_{ji}$ are real; in higher dimensions we consider series of spherical harmonics.
We will obtain conditions on the coefficients  $a_{ji} $ which imply that $u$ is in $h^\infty_v(\mathbf B)$ almost surely.
Our estimate improves previous results by Bennett, Stegenga and Timoney, and we prove that the estimate is sharp. 
The results for growth spaces can easily be applied to Bloch-type spaces, and we obtain a similar characterization for these spaces, which generalizes results by Anderson, Clunie and Pommerenke and by Guo and Liu.  
\end{abstract}

\maketitle 

\section{Introduction}
\subsection {Spaces of harmonic functions} 
Let $v$ be a positive increasing continuous function on $[0,1)$, assume that $v(0)=1$ and $\lim_{r\rightarrow 1} v(r)=+\infty$. We study  growth spaces of harmonic functions in the unit disk $\D $ and also in the multidimensional unit ball $\B $ in $\R^n $. We denote
\begin{equation*}
h^\infty_v (\D)=\{u:\D\rightarrow\R, \Delta u=0, |u(x)|\le Kv(|x|)\ {\text{for some}}\ K>0\},
\end{equation*}
and  $\hvb $ is defined similarly.
The study of harmonic growth spaces on the disk and the corresponding spaces of analytic functions $A_v^\infinity $ was  initiated  
by L. Rubel and A. Shields in \cite{RS} and by A. Shields and D. Williams in \cite{SW1,SW}.
Recently multidimensional analogs were considered in \cite{AD,EM}. Various results on the coefficients of functions in growth spaces were obtained in \cite{BST}. Hadamard gap series in growth spaces have been studied by a number of authors, see \cite {E} and references therein.

Examples of functions in $\hv $   can be constructed by lacunary series, see \cite {E}.  Another way to construct examples is by using random series, and such functions will be the
main focus of this paper. 
We consider
\beq\label {u}
u (re^{i\ta},\x) = \sum_{j=0}^\infty (a_{j0} \x_{j0} r^j \cos j\ta +a_{j1} \x_{j1} r^j \sin j\ta)
\eeq
 where $\x=\{\x_{ji}\}
 $  is a  sequence of  independent random variables and $\fa_j := (a_{j0},a_{j1}) \in \R^2$. 
  We will also study random harmonic functions on $\B$; 
 such functions can be written as
\begin {equation}
\label {un}
u(x,\x) =\sum_{m=0}^\infty \sum_{l=0}^{L_m} a_{ml} \x_{ml} r^{m} Y_{ml} \left(\fraction {x} {r}\right) 
\end {equation}
where $r = |x | $, $\{L_m\}$ depends on $n $  and $Y_{ml} $  are spherical harmonics of 
degree $m $ normalized to fulfill $| |Y_{ml} | |_\infinity\le 1 $. 
Our main results will be proven in several dimensions.

We always assume that the
 weights  satisfy the doubling condition  \begin{equation}
\label{vdouble}
v(1-d)\le Dv(1-2d).
\end{equation} 
Typical examples  are $v (r)=\left(\fraction1 { 1 -r }\right)^\alpha $ and $v (r)=
 \max\left\{1,\left(\log\fraction1 { 1 -r }\right)^\alpha \right\}$ for $\alpha>0 $.   
  For convenience we define a new function $g:[1,\infty)\rightarrow[1,\infty)$ such that $g(x)=v(1-\fraction1x)$. Then (\ref{vdouble}) is equivalent to
\begin{equation}
\label{gdouble}
g(2x)\le Dg(x).
\end{equation}
We will use $v$ and $g $ interchangeably.

The Bloch space is the space of analytic functions $f $ on $\D$ satisfying $$ |f (0) |+ \sup_{z\in\D}(1 - |z |^2)|f' (z)|   <\infty. $$
The generalizations of this space where $1 - |z |^2$ is replaced by another weight $w(|z|) $ which is decreasing and fulfills $\lim_ {r\rightarrow 1^-} w(r)=0$ are called Bloch-type spaces.
A harmonic function $u$ is in the Bloch-type space $\bl_w$ if
\begin {equation*}
| |u ||_{\bl_w} =  |u (0) |+\sup_{z\in\D} w(|z|)|\nabla u (z)|  <\infty. 
\end {equation*}
Random Bloch functions have been studied by J. M. Anderson, J. Clunie and Ch.~Pommerenke in \cite {ACP} and F. Gao in \cite {G}.


\subsection {Known results}
Let $\fa_j = (a_{j0},a_{j1}) \in \R^2$ and  $|\fa_j |=(|a_{j0}|^2+|a_{j1}|^2)^{1/2}$. 
It is not difficult to show that if $u(re^{i\ta}) = \sum_{j=0}^\infty (a_{j0}  r^j \cos j\ta +a_{j1} r^j \sin j\ta) \in h_v^\infty(\D)$, then 
\beq 
\label {square sum one}
\sum_{j=0}^{n}|\fa_j|^2\le Bg( n)^2,
\eeq
see for example \cite {BST}.
On the other hand, the inequality
\beq
\label {absolute values}
\sum_{j=0}^{n}|\fa_j|\le Bg( n)
\eeq
is sufficient to imply that $u\in \hv $, but it is not necessary. In the special case of Hadamard gap series, \eqref{absolute values} is both necessary and sufficient, see \cite{E}, and this is also the case when all the coefficients are positive 
\cite {BST}. But it is not possible in general to characterize all functions in $\hv $ by the absolute value of their coefficients.
We will obtain conditions on the coefficients  which imply that $u$ defined by \eqref{u} is in $\hv $ almost surely, and similarly in higher dimensions.

Let the partial sums of $u (re^{i\ta}) = \sum_{j=0}^\infty (a_{j0}  r^j \cos j\ta +a_{j1} r^j \sin j\ta) 
$ be denoted as follows
$$(s_n u)(re^{i\ta})=\sum_{j=0}^{n -1}(a_{j0}  r^j \cos j\ta +a_{j1} r^j \sin j\ta) 
 $$
 and denote the corresponding Ces\`{a}ro means by
$$(\sigma_n u)(re^{i\ta})=\fraction1n \sum_{j=0}^{n- 1}  (s_ju) (re^{i\ta}) 
= \sum_{j=0}^{n- 1} \left(1-\fraction{j}n\right) (a_{j0}  r^j \cos j\ta +a_{j1} r^j \sin j\ta) 
.$$
By Theorem 3.4 in \cite[p. 89] {Z}, the maximum of the Ces\`{a}ro means is less than or equal to the maximum of the function,
\beq
\label{ces}
\max_{\ta
} |u (re^{i\ta})
|\ge\max_{\ta} |(\sigma_n u)(re^{i\ta})|\qquad\text { for every $n$}.   
\eeq
Although functions in $\hv $ cannot be characterized by the coefficients alone, they can be characterized by 
their Ces\`{a}ro means, the following is Theorem 1.4 in \cite{BST}:
\begin {facta} 
\label {bst}
Assume $v $ satisfies \eqref{vdouble}. 
If $u $ is a harmonic function on the unit disk, then $u\in\hv $ if and only if $\|\sigma_n u\|_\infinity \le Cg (n) $ for all $n\ge 1 $ and some constant $C\ge 0 $. 
\end {facta}
If we consider the partial sums instead, then $u\in \hv $ only implies that $$\|s_n u\|_\infinity\le C g (n)\log n,$$ and this result is sharp,  see \cite {BST}.

Random Taylor series is a fascinating subject in harmonic analysis, we refer the reader to \cite{Ka} for an excellent introduction sixto the subject and further references.  One of the central results that we use goes back to R.~Salem and A.~Zygmund \cite{SZ}; it gives an estimate for the distribution function of a random polynomial. In \cite{SZ} trigonomertic  polynomials of the form $\sum_{j=0}^N\x_ja_j\cos j\phi$ are considered, where $\x_j$ is a Rademacher sequence (a sequence of independent random variables which take the values $1 $ and $-1 $ with equal probability) or a Steinhaus sequence (a sequence $\{e^{i\varphi_j}\}$ where $\varphi_j $ are 
independent and have uniform distribution in $[0,2\pi]) $. In \cite{Ka} the corresponding result is generalized to other series and  subnormal random sequences (which include both Rademacher and  Gaussian sequences and the real part of Steinhaus sequences).  

Conditions on the coefficients of random Taylor series of analytic functions in various functions spaces have been studied previously in \cite{ACP} and \cite {BST}. 
In  \cite{ACP} J.~M.~Anderson, J.~Clunie and Ch.~Pommerenke showed 
that if $c_j\ge 0 $, $\{e^{i\varphi_j}\}$   is a 
Steinhaus sequence and
\beq
\label {bloch}
\left(\sum_{j=0}^n j^2 c_j^2\right)^{1/2}=O\left (\fraction {n} {\sqrt{\log n}}\right), 
\eeq
 then $f (z, \varphi) =\sum_{j=0}^\infty c_j e^{i\varphi_j} z^j$ belongs to the Bloch space almost surely. 

F. Gao characterized Bloch functions for the case where 
the random sequence is a Rademacher sequence; the results give necessary and sufficient conditions  for 
a function to be a Bloch function almost surely,
 see \cite {G}. The conditions are given in terms of non-decreasing rearrangements.

Let $A_v^\infinity $ denote the space of analytic functions which fulfill $|u(z)|\le Kv(|z|)$ for some $K $. 
In \cite {BST} 
G. Bennett, D. A. Stegenga and R. M. Timoney proved the following: 
\begin{factb}
\label{bst2}
If $\{c_j\}_{j =0}^\infinity $ is a sequence satisfying
\beqs
\left(\sum_{j=0}^n |c_j|^2\right)^{1/2} \le C\fraction {g (n)} {\sqrt{\log n}}, 
\eeqs
and $\{e^{i\varphi_j}\}_{j=0}^\infinity $   is a Steinhaus sequence,  
then $\sum_{j=0}^\infty  c_j e^{i\varphi_j} z^j \in A_v^\infinity$ almost surely. 
\end{factb}

\subsection {Contents and organization of this paper}
In this paper we consider random functions given by \eqref {u} or more generally by  \eqref {un} with a random subnormal sequence $\xi_{ml}$.  
The reason for considering subnormal sequences is that they include both Rademacher and normalized Gaussian sequences, and the proofs are based only on the fundamental inequality $\mathcal E(e^{\lambda \xi})\le e^{\lambda^2/2}$ that is used to define subnormal sequences.

The main result of the paper is a sufficient condition on the coefficients $\{a_{ml}\}$ under which the random series \eqref {un} belongs to $h^\infty_v(\B)$ almost surely. As a consequence of this result we obtain a generalization of Theorem B to harmonic functions of several variables. In dimension 2 our main result is similar to Theorem B, but instead of summing all coefficients from $0 $ to $n $, we sum coefficients between $n_{k -1} $ and $n_k $ for some sequence $n_k $ that depends on $g$.  In this way we obtain results also in the case when $g $ grows more slowly than $\sqrt{\log x} $. 

Usually we start with a weight $v$ and ask for conditions on the coefficients $a_{ml}$ that guarantee that  the function defined by \eqref {un} is in $h_v^\infty$ almost surely. Another way to look at the result is by starting with a sequence of coefficients $\{a_{ml}\}$ and asking for the correct order of growth of typical functions given by  \eqref {un} .  We give some examples and show that in some cases our main result gives a better (more slowly growing) estimate than Theorem B.

In section \ref {motivation} we collect necessary definitions and preliminary results, and we also formulate a statement which illustrates how adding randomness to the coefficients influences the growth of the function.  The main result and some corollaries are given in section 3. In section 4 we show that  the main result is sharp (in some sense). We also prove some necessary conditions on the coefficients of functions in $\hv $ in section \ref {necessary section}. Our results can be applied to random functions in Bloch-type spaces and analytic growth spaces, and we obtain similar results for such functions in section~\ref {application}.

\section{Motivation and preliminaries} 
\label {motivation}
\subsection {Subnormal variables}
We will now consider random functions given by \eqref {u} and \eqref {un} where $\x=\{\x_{ji}\}$  is a  sequence of random variables. 
We will restrict ourselves to subnormal variables. 
 
\begin {definition}
A real-valued random variable  $\omega $ is called {\it subnormal} if
$$\mathcal {E} (e^{\lambda \omega})\le e^{\lambda^2/2}\qquad \mathrm{for\, all}\quad-\infinity <\lambda <\infinity.$$
A sequence of independent subnormal variables is called a {\it subnormal sequence}.
\end {definition}

The random variable that takes the values $1 $ and $-1 $ with equal probability is subnormal since  $\E(e^{\lambda \omega}) =\fraction12 (e^{\lambda} +e^{-\lambda})\le e^{\fraction12\lambda^2} $. A Rademacher 
sequence is the sequence of independent variables with such a probability distribution,  thus it is a subnormal sequence. 
Any real random variable $\omega $ with $\E (\omega) = 0$ and $|\omega | \le 1$ a.s. is subnormal.
 A Gaussian normal variable is subnormal if $\E (\omega) = 0$ and $Var (\omega) \le 1 $; see \cite[p. 67] {Ka} and \cite [p. 292]{Strom} for more on subnormal variables.
 
 Unlike Rademacher and Steinhaus variables, subnormal  variables are not necessarily symmetric.

\subsection {Deterministic and random series in growth spaces}
The result below illustrates that the random  sequence influences the growth of the function. If the growth restriction on the coefficients is strong enough, we can get a result that implies that the function is in $\hv $.
Another assumption implies that the function is in $\hv $ almost surely.  The last point of the proposition concerns a function with large (carefully chosen) coefficients, for which the choice 
 of signs still makes the function belong to $\hv $. The coefficients are large in the sense that $\sum_{j=0}^n a_j^2\ge Cg (n)^2 $ for  some $C$, and this is as large as they can be according to \eqref {square sum one}.

Let  $n_0 =1 $ 
and for some $A>1$ 
define $n_k $ by induction as 
\beq
\label {n}
n_{k+1}=\min \{l \in \N: g(l)\ge A g(n_k)\}.
\eeq
Choose $A $ large enough to make $n_k\ge2 n_{k-1} $. This way of defining a sequence  $\{n_k\} $ will be used several times.
In particular, if $v (r)=\left(\fraction1 { 1 -r }\right)^\alpha $ or $
  \max\left\{1,\left(\log\fraction1 { 1 -r }\right)^\alpha \right\}$, 
 we can choose $n_k=2^k$ and $n_k=2^{2^k} $, respectively.

\begin {prop}
\label {three things}
Let $u (re^{i\ta},\x) = \sum_{j=0}^\infty (a_{j0} \x_{j0} r^j \cos j\ta +a_{j1} \x_{j1} r^j \sin j\ta) 
$.\\
(i) If $|\fa_j |\le \fraction {g(n_k)} {n_k}$ for $n_{k-1} <j\le n_{k} $, then  $u (z,\x)\in \hv $ for all sequences $\{\x_{ji}\} $ with $\x_{ji}\in\{- 1,1\} $.\\
(ii) If $|\fa_j|\le\fraction {g(n_k)}{\sqrt{n_k\log n_k}} $ for $n_{k -1} <j\le n_{k} $ and $\{\x_{ji}\} $ is a subnormal sequence, then $u (z,\x) \in \hv $ almost surely.\\
 (iii) If $a_j=\fraction {g(n_k)}{\sqrt{n_k} } $ for $n_{k-1} <j\le n_{k} $, then there exists a  sequence $\{\x_j\} $  with $\x_j\in\{-1,1\} $ such that  $u (z,\x)= \sum_{j=0}^\infty a_j \x_j r^j \cos j\ta \in \hv $.\\
\end {prop}

\begin {proof}
Point (i) follows from \eqref {absolute values}, and (ii) will follow from Corollary \ref {induction}. 
The function in (iii) is constructed as in the proof of Theorem 1.12 (b) in \cite {BST};  we will write this function in the proof of Proposition \ref{liminf}. 
\end {proof}
In Proposition \ref {constant} we will see that (ii) is sharp.

\subsection {Preliminaries on higher-dimensional functions}
 We consider real-valued functions of $d+1$ real variables, $d\ge 1$. Let $F_n$ be the space of restrictions of polynomials on $\R^{d+1}$ of degree less than or equal to $n$ to the unit sphere $S^{d}$. Then 
the following Bernstein inequality 
\beq
\label {bs}
| | \nabla  P ||_\infinity \le n | |P | |_\infinity
\eeq
 holds for all $n $  and all $P\in F_n $, where 
 the gradient is evaluated tangentially to the sphere, see for example \cite [Theorem V]{oKe}. 
For trigonometric polynomials this is a well-known inequality by Bernstein. 

The next lemma will be used to prove our main result.

\begin {lem}
\label{several dimensions}
 Let $P_n\in F_n$, $M_n=\max_{S^{d}}|P_n|$ and  $\alpha\in(0,1)$. 
Then there exists a spherical cap of measure 
$C((1-\alpha)/n)^d $    in which $|P_n |\ge \alpha M_n $, and $C $ depends  on $d $. 
\end {lem}

\begin {proof}
 Let $\delta (y,\z) $ be the geodesic distance between two points $y $ and $\z $ on $S^{d}$. Then let $B (y,\p)= \{\z\in S^{d}: 
 \delta(y,\z) < \p\} $ be the spherical cap of radius $\p $ with center in $y $. It can be shown that for the $d$-dimensional surface measure of the cap 
 \beq
 \label{cap}
    | B (y,\p)| \ge C \p^{d},
  \eeq 
  where the constant depends on $d $.
  
  Let $y_0 $ be a point at which $|P_n |=M_n$, and let $y_1 $ be the closest point where $|P_n | =\alpha M_n $; there is nothing to prove if such a point does not exist. Just as in the proof of Lemma 4.2.3 in \cite{SZ}, we have
$$M_n(1 -\alpha) = |P_n(y_0)|-|P_n(y_1)|\le |P_n(y_0)-P_n(y_1)|\le \delta(y_0,y_1)\max |\nabla P_n|$$
and by \eqref{bs}, $ \delta(y_0,y_1)\ge (1-\alpha)/n $. 
Therefore, by \eqref{cap}, there exists a spherical cap of measure at least $C((1-\alpha)/n)^d $   in which $|P_n |\ge \alpha M_n $.
\end {proof}

The next result is Theorem 1 in \cite [p. 68] {Ka}, which we will need to prove our main result.
\begin {factc}
Let $E $ be a measure space with a positive measure $\mu $, and $\mu (E) <\infinity $. Let $F $ be a linear space of measurable bounded functions on $E $, closed under complex conjugation, and suppose there exists $\rho>0 $ with the following property: if $f\in F$ and $f $ is real, there exists a measurable set $I =I (f)\subset E $ such that $\mu (I) \ge \mu (E)/\rho $ and $|f (t) |\ge \fraction12 \|f\|_\infinity $ for $t\in I $. Let us consider a random finite sum
$$P =\sum \x_j f_j$$
where $\x_j $ is a subnormal sequence and $f_j\in F $. Then, for all $\kappa>2 $,
$$\pr (\|P\|_\infinity \ge 3 (\sum \| f_j\|^2_\infinity \log (2\rho\kappa))^{1/2})\le \fraction {2} {\kappa}.  $$ 
\end {factc}

\section{Main results}
\label {main}

\subsection{Sufficient conditions on the coefficients}
We consider harmonic functions defined by \eqref{un}, where $Y_{ml}$ are spherical harmonics of degree $m $ 
on the sphere $S^d $, and
we use the notation $\fa_m = (a_{m0},...,a_{mL_m}) $, so $|\fa_m |^2 =\sum_{l=0}^{L_m} |a_{ml}|^2 $.
We are now ready to prove the following:
\begin {thrm}
\label {improved d}
Let  $\x=\{\x_{ml}\}$  be a subnormal sequence. If there exists an increasing sequence $\{n_k\} $ of positive  integers such that  for all $k $ we have
$g (n_{k+1})\le C_1 g (n_k) $  and
$$\sum_{j=1}^k \sqrt{\left(\sum_{m=n_{j-1}+1}^{n_j}|\fa_m |^2 
\right)
\log n_{j}}\le C_2 g (n_{k}), $$
then $u(x,\x) =\sum_{m=0}^\infty \sum_{l=0}^{L_m} a_{ml} \x_{ml} r^{m} Y_{ml} \left(\fraction {x} {r}\right)\in\hvb $ almost surely.
\end {thrm}
 
In two dimensions $|\fa_m |^2 $ is just $|a_{m0}|^2+|a_{m1}|^2$, so the same assumptions imply that $u (re^{i\ta},\x) = \sum_{m=0}^\infty (a_{m0} \x_{m0} r^m \cos m\ta +a_{m1} \x_{m1} r^m \sin m\ta)\in\hv $ almost surely.

\begin {proof} 
Let $S_n (y,\x) = \sum_{m=0}^n \sum_{l=0}^{L_m} a_{ml} \x_{ml} Y_{ml} (y) $ where $y\in S^{d} $ and denote $M_n (\x)=\max_{ y\in S^{d}} |S_n (y,\x)| $. Let $j=j (N) $  be such that $n_{j-1} <N\le n_{j} $ and define $Q_N(y,\x)=S_{N} (y,\x)-S_{n_{j-1}} (y,\x)$ and $\M_N (\x)=\max_{ y\in S^{d}} |Q_N (y,\x)| $. 
 Since harmonic polynomials on the sphere fulfill \eqref{bs}, 
 by Lemma \ref{bs} there exists a spherical cap of measure 
$C(\fraction1 {2N})^d $    in which $|Q_N |\ge\fraction12 \M_N $, where $C $ depends  on $d $. 
Then we can apply Theorem C to $Q_N $ with $E =S^{d} $, $\mu $  the surface measure on $S^{d} $, $F$ the set of functions of the form $\sum_{m=0}^N \sum_{l=0}^{L_m} a_{ml} \x_{ml} Y_{ml} (y) $, 
$\kappa =2N^2$, and $\rho $  a constant which depends on $d $. 
Define 
 $$E_N =\left\{\x: \M_N(\x)\ge K_1  \sqrt{\sum_{m =n_{j -1} +1}^N |\fa_m|^2\log N}\right\}, $$ 
where $K_1 $ is a constant which is chosen large enough to make $3\sqrt {\log 2\rho \kappa} \le K_1 \sqrt{\log N}$.
 Then since $\sum_{N=1}^{\infinity} \mathcal {P} (E_N ) =\sum_{N=1}^{\infinity} 1/N^2 <\infinity $, we have by the Borel-Cantelli lemma (see for example \cite [p. 7] {Ka}) that 
 for almost all $\x $ there is a  $J =J(\x)$ such that
 $$\M_{N} (\x) \le K_1\sqrt{ \sum_{m =n_{j -1} +1}^N |\fa_m |^2 \log N}
  $$ 
 for $N\ge n_J $. 
   Fix $L $ and let $n_{k-1}< L\le n_{k}$. Then for $L> n_{J} $,
 \begin{align*}
 M_L(\x)&\le  M_{n_{J-1}}(\x) +\sum_{j=J}^{k -1}\M_{n_j}(\x)+\M_{L}(\x)\\
 &\le B_\x 
 +K_1\sum_{j=J}^{k }\sqrt{ \sum_{m =n_{j -1} +1}^{n_j} |\fa_m |^2 \log n_j} 
 \le B_\x +K_1 C_2 g (n_{k})\\ 
 &\le B_\x + C_3  g (n_{k-1})\le B_\x +C_3  g (L)
 \qquad \mathrm{for\, a.e.\, }\x.
 \end{align*}
 Let $B_\x $ be large enough to make the inequality  $M_L(\x)\le B_\x +C_3  g (L)$ hold also for $0<L\le n_{J} $, and also let $M_0(\x)\le B_\x $.
 Let $r = | x | $ and $y=x/|x | $. By summation by parts,
 \begin{gather*}
 \left|\sum_{m=0}^n \sum_{l=0}^{L_m} a_{ml} \x_{ml} r^m Y_{ml}  \left(\fraction {x} {r}\right) 
 \right| =\left|r^n S_n(y,\x) -  (1 -r)\sum_{k=0}^{n-1} S_k (y,\x)r^k\right|\\
 \le  r^n(C_3 g (n) +B_\x) +  (1 -r)\left(B_\x +\sum_{k=1}^{n-1}  (C_3g (k)+B_\x)r^k\right).
 \end {gather*}
Then because of the doubling condition 
we get
\beq\label {to infinity}
\left| \sum_{m=0}^\infinity \sum_{l=0}^{L_m} a_{ml} \x_{ml} r^m Y_{ml}  \left(\fraction {x} {r}\right)\right| 
\le  C_3 (1 -r)\sum_{k=1}^{\infinity} g (k)r^k+B_\x\qquad \mathrm{for\, a.e.\, }\x. 
\eeq
 Pick $N $ such that $1-\fraction1 {N-1}<r\le 1-\fraction1 {N}$. Then 
\beq
\label {to n}
(1 -r)\sum_{k=1}^{N} g (k)r^k\le (1 -r)g (N)\sum_{k=1}^{N} r^k\le  g (N) 
\eeq
and

\begin {equation}
\label{from n}
\begin {split}
(1 -r)\sum_{k=N+1}^{\infinity} g (k)r^k
 &= (1 -r)\sum_{j=0}^{\infinity}r^{2^{j}N}\sum_{i=1}^{2^jN} g (2^jN +i)r^i\\
&\le (1 -r)\sum_{j=0}^{\infinity}g (2^{j+1}N )r^{2^{j}N}\sum_{i=1}^{2^jN} r^i\\
&\le  g (N)\sum_{j=0}^{\infinity}	D^{j+1}\left[\left(1 -\fraction1N\right)^N\right]^{2^{j}}\le C_4 g (N).
\end {split}
\end {equation}
Here $C_4$ depends on $D$ only. Then by \eqref{to infinity}, \eqref{to n} and \eqref{from n}, $u \in\hvb $ almost surely.
 
\end {proof}

\begin {remark}
\label {the improvement}
If we had applied Theorem C to $S_n $ instead of $Q_n $ we could have obtained
\beqs
\max_{y\in S}
|S_n (y,\x)|\le C \sqrt{\sum_{m=0}^{n}|\fa_m |^2 \log n}+C_\x 
\qquad \mathrm{for\, a.e.\, }\x.
\eeqs
Then if
\beq
\label {old version}
\left(\sum_{m=0}^{n}|\fa_m |^2 
\right)^{1/2} \le C\fraction {g (n)} {\sqrt{\log n}}, 
\eeq
we would get by partial summation as above that $u\in\hvb $ almost surely, and this generalizes Theorem B.
But the approach in Theorem \ref {improved d} is better for two reasons.  First of all it makes sense even if $g $ grows more slowly than $\sqrt{\log n} $.
For some examples it also gives a better estimate, in the sense that when the coefficients are given and we want to estimate the correct order of growth of a function, Theorem~\ref {improved d} may give a more slowly growing estimate for $g $ than we get by using \eqref {old version}: 
Let $n_k =2^{2^k} $ for $k =0,1,...$ and define $ a_0 =
a_1 =a_2=0 $ and
$$a_j=\fraction {1}{\sqrt{n_k}},\qquad n_{k-1} <j\le n_k.$$
For 
$u (z,\x) = \sum_{j=0}^\infty a_j \x_j r^j \cos j\ta $
\eqref{old version}
gives $g (x) = (\log x\log\log x)^{1/2}$ since 
$$\sum_{j=0}^{n_N} a_j^2 =\sum_{k=0}^N \fraction  {n_k -n_{k -1}}{n_k} \simeq N+ 1\simeq  \log\log n_N,$$ but Theorem \ref {improved d} gives $g (x) = (\log x)^{1/2}$  
since 
\beqs
\sum_{k=1}^N \sqrt{ \left(\sum_{j =n_{k -1}+1}^{n_k} a_j^2\right)\log n_{k}}
 \simeq C\sqrt{\log n_N }.
 \eeqs
We will see in Proposition \ref{constant} 
 that $g (x) = (\log x)^{1/2}$ is the optimal 
 estimate for this function. 
 \end {remark}

\begin {cor}
\label {induction}
Let $\x=\{\x_{ml}\}$  be a subnormal sequence and define $\{n_k\} $ as in \eqref{n}.
If
$$\left(\sum_{m=n_{k-1}+1}^{n_{k}} \fa_m^2\right)^{1/2} \le C\fraction {g (n_{k})} {\sqrt{\log n_{k}}}, $$
then $u(x,\x) =\sum_{m=0}^\infty \sum_{l=0}^{L_m} a_{ml} \x_{ml} r^{m} Y_{ml} \left(\fraction {x} {r}\right)\in\hvb $ almost surely. 
\end {cor}

\begin {proof}
By the doubling condition  $g (n_{k}) \leq Dg (n_{k}/2)\leq DA g (n_{k-1}),  $
and since  
\begin{align*}
\sum_{j=1}^{k }\sqrt{\sum_{m=n_{j-1} +1}^{n_{j}} \fa_m^2\log n_{j}}
\le C_1g (n_{k})\sum_{j=1}^{k } \fraction {1} {A^{k -j}}\le C_2 g (n_{k}),
\end{align*}
the result follows from  Theorem \ref {improved d}.
\end {proof}

\begin {remark} 
Now it follows easily that Proposition \ref {three things} (ii) is true.
Functions with coefficients
$$ |\fa_j|\le\fraction {g(n_k)}{\sqrt{n_k\log n_k}},\qquad n_{k-1} <j\le n_k, $$
 are  in $\hv $  almost surely by Corollary \ref {induction}. 
\end {remark}

\begin {remark} 
\label {basis}
It is not necessary to assume that $\{Y_{ml}\} $ is a basis  in the proof of Theorem \ref {improved d}, we can use any combination of spherical harmonics.  We will need this fact when we apply our results to Bloch-type functions.
\end {remark} 


\section {Sharpness of results}
\label {sharpness section}
\subsection {Sharpness of Corollary \ref {induction}}
We  will now prove that Corollary \ref {induction}
is sharp by giving 
an example. We will first prove it in the two-dimensional case, and then indicate how it can be generalized to any dimension. 
The 
example is similar to the one given in the proof of
Theorem 1.18 (b) in \cite{BST}. 
We will use that
\beq
\label {essential}
\left\|\sum_{j=1}^n c_j  \cos ( N+4^j)\ta \right\|_\infinity\ge c
\sum_{j=1}^n |c_j|
\eeq
for any $N $ and some absolute constant $c>0 $. This can be shown by using Riesz products. Let $A$ be a constant such that
\beq
\label {condition on a}
\fraction {1} {A-1}\le \fraction {c}8
\eeq
where $c $ is the constant in \eqref {essential}. Let $n_0 =2 $, and for some $A $ that fulfills \eqref{condition on a} define $n_k $ by induction as in \eqref{n}.
We choose $A $ big enough to make $n_k\ge 4 n_{k-1}$ .

\begin {prop}
\label {sharpness}
Let $\{\nu_k\} $ be any sequence of positive numbers increasing to infinity 
and define $\{n_k\} $ as  in \eqref{n}.  
Then for the sequence
$\{a_j\} $ where
$$a_j=\nu_{k}\fraction { g(n_k)}{\log n_k},\qquad \mathrm{when}\; j= n_{k-1}+4^m, \quad 0\le m\le \log_4 \fraction {n_k}2, $$
and $a_j =0 $ otherwise, we have
\beqs
\left(\sum_{j=n_{k-1}+1}^{n_{k }} a_j^2\right)^{1/2} \le C\nu_{{k}}\fraction {g (n_{k})} {\sqrt{\log n_{k}}},
\eeqs
but  $u (z,\x) = \sum_{j=0}^\infty a_j\x_j r^j \cos j\ta
\notin \hv $ for any choice of sequence $\{\x_j\} $ where $\x_j =\pm 1 $.
\end {prop}

\begin {proof}
 Inequality \eqref{condition on a} implies 
\beqs
\sum_{k=1}^{N -1} \nu_{k} g(n_k)\le \fraction c{8}\nu_{N}g(n_N).
\eeqs
 Let $\sigma_n$ be the Ces\`{a}ro mean, then by \eqref{essential} we have 
 for $n =n_N$,
\begin {gather*}
| |\sigma_n u | |_\infinity=\left\|\sum_{k =1}^{N}\nu_{k} \frac{g(n_k)}{\log n_k}\sum_{m = 0}^{\lfloor\log_4 (n_k/2)\rfloor} \left(1-\fraction {n_{k-1}+4^m}{n_N}\right)\x_{n_{k-1}+4^m} \cos ( n_{k-1}+4^m)\ta\right\|_\infinity\\
 \ge \nu_{N} \frac{g(n_N)}{\log n_N}\left\|\sum_{m = 0}^{\lfloor \log_4 (n_N/2)\rfloor} \left(1-\fraction {n_{N-1}+4^m}{n_N}\right) \x_{n_{N-1}+4^m} \cos ( n_{N-1}+4^m)\ta \right\|_\infinity \\
 - \left\|\sum_{k =1}^{N-1}\nu_{k} \frac{g(n_k)}{\log n_k}\sum_{m = 0}^{\lfloor\log_4 (n_k/2)\rfloor} \left(1-\fraction {n_{k-1}+4^m}{n_N}\right)\x_{n_{k-1}+4^m} \cos ( n_{k-1}+4^m)\ta\right\|_\infinity\\
\ge 
c\fraction1{4\log 4} \nu_{N} g(n_N) -\fraction1 {\log 4}\sum_{k =1}^{N-1}\nu_{k} g(n_k) \\
\ge \fraction1 {\log 4}\left(\fraction{c}4 -\fraction{c} {8}\right)  \nu_{N} g(n_N)= C  \nu_{N} g(n_N)
\end {gather*}
Hence by Theorem A we get that $u (z,\x) 
\notin \hv $.
\end {proof}

To prove the same in $\R^{d +1} $, let $Y_{j0}(y)=\Re (y_1+iy_2)^j=\cos j\ta$, where $y = (y_1,...,y_{d +1})$ and $\ta =\arctan \fraction {y_2} {y_1} $.  Also let $a_{j0}=a_j$, where $a_j $ is as above, and $a_{ji}=0$ otherwise. 
Then $u (x,\x)= \sum_{j=0}^\infty  a_{j0} \x_{j0} r^j Y_{j0} (\frac{x}{r}) \notin \hvb$. 

\subsection {Sharpness of Proposition \ref{three things} (ii)} 
The next example serves two purposes, one is to prove in another way that Corollary \ref {induction} is sharp, the other is to show that the estimate in Proposition \ref {three things} (ii) cannot be improved.
 
To construct this example
we need a result which is based on Lemma 4.5.1 in \cite {SZ}. This lemma is used in a similar way in
Theorem 3.7 in \cite{ACP} to prove a result on the coefficients of Bloch functions.
\begin {factal}
\label {sharp lemma}
Let
$\x=\{\x_k\}_{k=0}^\infty $  be a Rademacher sequence.  Let
$$H_n (\ta,\x) =\sum_{j=0}^{n}b_{j}\x_j  
\cos j\ta$$
 and
$$R_n =\sum_{j=0}^{n}b_{j}^2,\qquad T_n =\sum_{j=0}^{n}b_{j}^4 \le c\,\fraction  {R_n^2} {n}.$$
Then
\beqs
\max_{\ta} |H_n(\ta,\x)|>C\sqrt{R_n\log n_n}\qquad (C>0) 
\eeqs
except for $(\x_0,\x_1,...,\x_{n})\in E_n $ where $\pr(E_n)<B (c)\,n^{-1/10}$. The constant $C $ is absolute and $B $ depends on $c $.
\end {factal}

Then we have

\begin {prop}
\label {constant}
Let  $\x=\{\x_j\}_{j=0}^\infty $  be a Rademacher sequence,
let   $n_0 =1 $ 
and for some $A$ large enough define $\{n_k\} $ by induction as 
in \eqref{n}. Let $\{\nu_k\} $ be any sequence of positive numbers increasing to $\infinity $. 
Then for the sequence 
$\{a_j\} $ where
$ a_0 =
a_1 =a_2=0 $ and
$$a_j=\nu_{{k}}\fraction {g(n_k)}{\sqrt{n_k\log n_k}},\qquad n_{k-1} <j\le n_k, $$
we have
\beqs
\left(\sum_{j=n_{k-1}+1}^{n_{k }} a_j^2\right)^{1/2} \le \nu_{k}\fraction {g (n_{k})} {\sqrt{\log n_{k}}},
\eeqs
but  
almost surely $u (z,\x) = \sum_{j=0}^\infty a_j\x_j r^j \cos j\ta\notin \hv $.
\end {prop}

The main difference between the proof of Theorem 3.7 in \cite {ACP}  and the proof of this   result lies in the fact that we need to make it hold for slow growing weights as well, and we split the function $u$ in two parts which are estimated separately. Lemma A is applied to only a part of the function.

\begin {proof} The constants  $C_j $, $j =1,2... $ in this proof will be absolute constants.
Define the sequence $\{n_k\} $ by induction as stated, where we choose
$A \ge 2$ and such that
the following condition is satisfied:
\beq
\label{three conditions}
n_k>2n_{k-1}.\eeq
One more condition on
$A $ will be specified later.
 
Fix $r_N =1 -1/n_N $ and split $u $ into two parts
\begin{align*} u(r_Ne^{i\ta},\x)= \sum_{j=0}^\infty a_j\x_j r_N^j \cos j\ta
&=\sum_{j=0}^{n_{N -1}} a_j\x_j r_N^j \cos j\ta +\sum_{j=n_{N -1}+1}^\infty a_j\x_j r_N^j \cos j\ta\\
&=b_N (r_Ne^{i\ta},\x) +d_N (r_Ne^{i\ta},\x).
\end{align*} 
Then
\beq
\label {db} |u(r_Ne^{i\ta},\x)|=\left|\sum_{j=0}^\infty a_j\x_j r_N^j \cos j\ta\right|\ge  |d_N (r_Ne^{i\ta},\x)|-|b_N (r_Ne^{i\ta},\x)|
\eeq
We will estimate $|d_N (r_Ne^{i\ta},\x)| $ from below and $|b_N (r_Ne^{i\ta},\x)| $ from above.
Let $$h_N(\ta,\x)
=\sum_{j=n_{N -1} +1}^{n_N}\left (1 -\fraction {j}{n_N}\right) a_j\x_j r_N^j \cos j\ta$$ 
This is the Ces\`{a}ro mean of the partial sum of $d (r_Ne^{i\ta},\x)$.
By \eqref{ces}, 
\beq\label {dh}
\max_{\ta
} |d (r_Ne^{i\ta},\x) 
|\ge\max_{\ta} |h_N(\ta,\x)|. 
\eeq
We will apply Lemma A
to $h_N $.
Using 
\eqref{three conditions}, we get
\begin{align*}
R_{n_N}
&=\sum_{j=n_{N-1} +1}^{n_N} \left (1 -\fraction {j}{n_N}\right)^2 a_j^2 r_N^{2j}
\ge C_1\sum_{j=n_N/2+1}^{3n_N/4} \left (1 -\fraction {j}{n_N}\right)^2 a_j^2 \\
&\ge C_1\fraction {n_N}{4}\left(\fraction14\right)^2\fraction {\nu_{N}^2g(n_N)^2}{n_N\log n_N}
\ge C_2\fraction  {\nu_{N}^2g(n_N)^2}{\log n_N}.
\end{align*}
Furthermore,
\begin{align*}
T_{n_N} =\sum_{j=n_{N-1} +1}^{n_N} \left (1 -\fraction {j}{n_N}\right)^4 a_j^4 r_N^{4j}
\le \fraction {(n_N-n_{N-1})\nu_{N}^4 g(n_N)^4}{(n_N\log n_N)^2} \le C_3\fraction  {R_{n_N}^2} {n_N}.
\end{align*}
Then by  Lemma A,
\beq
\label {below}
\max_{\ta } |h_N(\ta,\x)
|>C_4\sqrt{R_{n_N}
\log n_N}\ge C_5\nu_{N}g(n_N)
\eeq
except for $\x \in E_{n_N}$.
Since $\sum_{k=1}^{\infinity} \pr (E_{n_k} )< \sum_{k=1}^{\infinity}B(C_3)\, n_k^{- 1/10}$, 
and this is finite by \eqref{three conditions}, we have by the Borel-Cantelli lemma that for almost all $\x $ there exists a $N_0 =N_0 (\x)$ such that \eqref{below} holds for all $N\ge N_0 $.
 Hence by \eqref{dh}, for almost all $\x $ we have for $N\ge N_0(\x) $ that
 \beq\label {d}
\max_{\ta} |d (r_Ne^{i\ta},\x)|
 \ge C_5\nu_{N}g(n_N).  \eeq

Let $S_n (\ta,\x) = \sum_{k=0}^n a_k \x_k \cos k\ta $ and $M_n (\x)=\max_{0 \le \ta\le 2\pi} |S_n (\ta,\x)| $. Let $j=j (n) $  be such that $n_{j-1} <n\le n_{j} $ and define $Q_n(\ta,\x)=S_{n} (\ta,\x)-S_{n_{j-1}} (\ta,\x)$ and $\M_n (\x)=\max_{0 \le \ta\le 2\pi} |Q_n (\ta,\x)| $. 
Just as in the proof of Theorem \ref {improved d}, it can be shown that
 for almost all $\x $ there is   $J =J(\x)$ such that 
\beqs
\M_{n} (\x) \le 
 K_1\sqrt{ \left(\sum_{l=n_{j-1}+1}^{n}  a_l^2 \right)\log n_{j}}
\le K_1\nu_{{j}}g(n_{j})
\eeqs
 for $n\ge n_J $. 
    Fix $L $ and let $n_{k-1}< L\le n_{k}$. 
  Then for a.e. $\x$ and $L\ge n_J(\x)$, 
\begin{equation}\label {m}
\begin{split}
 M_L(\x)&\le  M_{n_{J-1}}(\x) +\sum_{j=J}^{k -1}\M_{n_j}(\x)+\M_{L}(\x) \le B_\x
 +K_1\sum_{j=J}^{k  }\nu_{j}g(n_{j})\\
 &\le B_\x +K_1\nu_{k}g(n_{k })\sum_{l=0}^{k-J}  \fraction1 {A^l}
 \le
  B_\x +2K_1 \nu_{k}g(n_{k }).
 \end{split}
\end{equation}
Let $B_\x $ be large enough to make the inequality  $M_L(\x)\le B_\x +2K_1  g (L)$ hold also for $0<L\le n_{J} $, and also let $M_0(\x)\le B_\x $. 

We will now 
estimate $b_N(r_Ne^{i\ta},\x) $. 
   By summation by parts 
   and \eqref{m},
 \begin{gather*}
 \left|b_N(r_Ne^{i\ta},\x)\right| =\left|\sum_{l=0}^{n_ {N-1}} a_l \x_l r_N^l\cos l\ta\right| \\
 =\left|r_N^{n_ {N-1}} S_{n_ {N-1}}(\ta,\x) -  (1 -r_N)\sum_{l=0}^{n_ {N-1}-1} S_l(\ta,\x)  r_N^l\right|\\
 \le  r_N^{n_{N-1}}M_{n_{N-1}}(\x)  + (1 -r_N)\left(B_\x+ 
 \sum_{j=0}^{N-2}  (B_\x +2K_1 \nu_{j}g(n_{j})) \sum_{l=n_{j}}^{n_{j+1}-1}r_N^l\right)
 \\
 \le  (2K_1 \nu_{N-1}g(n_{N-1}) +B_\x) +B_\x + 2K_1\nu_{N-1}g(n_{N-1}) \sum_{j=0}^{N-2}  \fraction1 {A^j}
 \end {gather*}
  Then
  \beq
  \label {b}
  \max_{\ta } |b_N(r_Ne^{i\ta},\x)|\le 2B_\x +6K_1  \nu_{N-1}g(n_{N-1})  \qquad \mathrm{for\, a.e.\, }\x.    \eeq

 For almost every $\x$ and $N \ge J(\x)$  we get by letting
$A\ge 12K_1/C_5 $ and using \eqref{db}, \eqref{d} and \eqref{b} that
\begin{align*}
\max_{\ta } |u(r_Ne^{i\ta},\x)|
& >C_5\nu_{N}g(n_{N})-6K_1\nu_{N-1}g(n_{N-1})-2B_\x\\
&\ge C_5\nu_{N}g(n_{N})-\fraction {6K_1} {A}\nu_{N-1}g(n_{N})-2B_\x\\
&\ge \fraction  {C_5}2 \nu_{N}g(n_{N}) -2B_\x
=\fraction  {C_5}2\nu_{N}v(r_N) -2B_\x. 
\end{align*}
 Then almost surely $u (z,\x) = \sum_{j=0}^\infty a_j \x_j r^j \cos j\ta\notin \hv $.

\end {proof}


\section{Some results for deterministic functions}
\label {necessary section}
\subsection{Necessary conditions on a general function in $\hv $} 
We will now prove some estimates for the growth of the coefficients of functions in $\hv $.
We know that $ |\fa_j|\le  C g (j) $ from for example \eqref{square sum one}. For Hadamard gap series there exist examples  of functions in $\hv $ 
for  which
$$\limsup_{j\rightarrow \infinity}\fraction {| \fa_j |}{g (j)}>0,$$
for example $u (z) =\sum_{k=0}^\infty g(n_k) r^{n_k}\cos n_k \ta$ where $\{n_k\}$ is defined by \eqref {n}, see \cite{E}. But all the coefficients cannot grow this fast if $u\in\hv$: 
\begin {prop}
\label {liminf}
Let $u (re^{i\ta}) = \sum_{j=0}^\infty (a_{j0}  r^j \cos j\ta +a_{j1}  r^j \sin j\ta)
\in\hv $ and 
define a sequence $\{n_k\} $ as before. Let $k=k (j) $  be such that $n_{k-1} <j\le n_{k} $. 
Then
\beq\label {the limit}
\liminf_{j\rightarrow \infinity} \fraction {| \fa_j |\sqrt{n_{k}}} {g (j)}<\infty.  
\eeq
Moreover, there exists a function  in $\hv $ for which $\liminf_{j\rightarrow \infinity} | \fa_j |\sqrt{n_{k}}/ {g (j)}>0 $, 
so the result  is sharp.
\end {prop}

A related result is given in Theorem 1.16 (a) in \cite{BST}, there it is proven that if $u (z) =\sum_{j=0}^\infinity b_j z^j\in A_v^\infinity $ and $|b_n | $ increases with $j $, then $|b_j|=O(g(j)/\sqrt{j}) $.

When $g $ grows like $x^\alpha $ it would be equivalent to replace  $n_k $ in \eqref{the limit} by $j $, but for slow-growing functions like $\log x $, that would give a weaker statement since $n_k $ in that case grows very fast.

\begin {proof}
In Theorem 1.12 (b) in \cite {BST} it is proven that 
\beq
\label {sq}\sum_{j=0}^n |\fa_j|\le C g(n)\sqrt{n}
\eeq
whenever $u\in\hv $.  Then since $n_k\ge 2n_{k-1} $, 
$$\fraction {n_k}2\min_{j\in (n_{k -1}, n_k]} |\fa_j|\le  \sum_{j=n_{k -1}+1}^{n_{k}} |\fa_j|\le C  g(n_k)\sqrt{n_k},$$
thus $$\min_{j\in (n_{k -1}, n_k]} |\fa_j|\le 2C  g(n_k)/\sqrt{n_k}\le 2ADCg(j)/\sqrt{n_k}$$ and 
where $D $ and $A $ are as in \eqref{gdouble} and \eqref{n}, respectively, and the result follows.

The function used in \cite {BST} to prove that Theorem 1.12 (b) is sharp can also be used here. To construct this function, it is used that  there exists a sequence $\{\x_j\}$ in $\{-1,1\} $ such that the polynomials $$P_m (z) =\fraction {\sum_{j=1}^m \x_j z^j} {\sqrt{m}} $$
satisfy $| |P_m | |_\infinity\le 5 $, see \cite{R}. These are called 
Rudin-Shapiro polynomials. Now define
$$u(z) =\Re\sum_{k=1}^\infinity g (n_k)z^{n_{k-1}}P_{n_k -n_{k -1}}(z) = \sum_{k=1}^\infinity \fraction {g (n_k)r^{n_{k-1}}} {\sqrt{n_k-n_{k -1}}}\sum_{j=1}^{n_k-n_{k-1}}\x_{j}r^j\cos (n_{k-1}+j) \ta.$$
By \eqref{ces} we have $| |\sigma_n u| |_\infinity \le | |s_n u | |_\infinity $, so $u\in\hv $ by Theorem A since $| |s_nu | |_\infinity\le Cg (n) $. The coefficients have the desired growth since $n_k\ge 2 n_{k-1} $.
\end {proof}
The function constructed in the above proof also proves Proposition \ref {three things} (iii).

The estimate $|\fa_j|\le p_j g(j)/\sqrt{j} $, where $\{p_j\} $ is a sequence going to infinity, holds for most of the coefficients. More precisely:

\begin {prop}
\label {fraction}
Assume $u (re^{i\ta}) = \sum_{j=0}^\infty (a_{j0}  r^j \cos j\ta +a_{j1} r^j \sin j\ta)
\in \hv$
and let $p_j $ be an increasing sequence of positive numbers such that $\lim_{j\rightarrow \infinity} p_j =\infinity $.
 Define $N(n) $ as the number of $\fa_j $ satisfying $j\le n $ and $|\fa_j |\le p_j g(j)/\sqrt{j}$. Then
$$\lim_{n\rightarrow \infinity} N (n)/n =1.$$
\end {prop}
A similar result was proved by F. G. Avhadiev and I. R. Kayumov in \cite {AK} for Bloch functions by a different argument.
\begin {proof}
Let $I_k =|\{j:2^{k-1}<j\le 2^{k}, |\fa_j|> p_j g(j)/\sqrt{j}\}|$. Since by \eqref {sq} we have
$$I_k p_{2^{k-1}} g(2^{k-1})/\sqrt{2^{k}}<  \sum_{j=2^{k -1}+1}^{2^{k}} |\fa_j|\le C  g(2^{k})\sqrt{2^{k}},$$
it follows that $I_k <DC 2^{k}/p_{2^{k-1}} $. If $2^{m-1} <n\le 2^m $, then 
$$N(n) \ge n -\sum_{k=1}^m I_k= n -o (n).$$
\end {proof}

\section {Application to other spaces}
\label {application}
\subsection {Bloch-type spaces} 
We will now see that our results for growth spaces can easily be applied to Bloch-type spaces $\bl_w$. We will consider these spaces in several dimensions, and they are defined as the spaces of functions that fulfill
$$| |u ||_{\bl_w} =  |u (0) |+\sup_{z\in\B} w(|z|)|\nabla u (z)|  <\infty, $$
where $w $  is as described in the introduction.
We always assume that $w $ fullfills a condition equivalent to \eqref{vdouble}:
\beq
\label {wdouble}
 w\left(1-\fraction {d} {2}\right) \ge B w(1-d).
\eeq
Examples are $w (r)=( 1 -r )^{\alpha }$ and $
 (\log \fraction1 { 1 -r} )^{-\alpha} $ for $\alpha>0 $.

The function $u(x,\x) =\sum_{m=0}^\infty \sum_{l=0}^{L_m} a_{ml} \x_{ml} r^{m} Y_{ml} \left(\fraction {x} {r}\right)$ is in $\bl_w $ 
 if and only if all partial derivatives of $u$
are in  $\hvb $ for $v (r) =1/w (r) $. We can write $Y_{ml}(x)$ instead of  $r^{m} Y_{ml} \left(\fraction {x} {r}\right)$, and $Y_{ml} (x) $ is a homogeneous harmonic polynomial. By Theorem III in \cite {oKe} we have $|\fraction {\partial} {\partial x_i}Y_{ml} \left(x\right)| \le m $.
Then
\begin {equation*}
\fraction {\partial} {\partial x_i}u(x,\x) 
 =\sum_{m=1}^\infty \sum_{l=0}^{L_m} a_{ml} \x_{ml}  \fraction {\partial} {\partial x_i}Y_{ml} \left(x\right) 
 =\sum_{m=1}^\infty \sum_{l=0}^{L_m}m a_{ml} \x_{ml} \fraction {\partial} {\partial x_i}\fraction {Y_{ml} \left(x\right) } {m}
 \end {equation*}
By Remark \ref {basis} and since $\fraction {\partial} {\partial x_i}\fraction {Y_{ml} \left(x\right) } {m}  $  is a homogeneous harmonic polynomial bounded by $1 $ on the sphere, we can apply  Theorem \ref {improved d} with $w (r) =1/v (r) $.
Then the next result generalizes \eqref{bloch}
to all weights that satisfy \eqref{wdouble}. It also generalizes Theorem 1 by J.~Guo and P.~Liu in \cite {GL}, which is proved for $\alpha $-Bloch functions.
\begin{cor}
\label{random improved}
Let $u(x,\x) =\sum_{m=0}^\infty \sum_{l=0}^{L_m} a_{ml} \x_{ml} r^{m} Y_{ml} \left(\fraction {x} {r}\right) $ 
where $\x=\{\x_{ml}\}
$  is a subnormal sequence. If there exists an increasing sequence $\{n_k\} $ of positive  integers such that  for all $k $ we have
$g (n_{k+1})\le C_1 g (n_k) $  and
$$\sum_{i=1}^k \sqrt{\left(\sum_{m=n_{i-1}+1}^{n_i} m^2\fa_{m}^2 \right)\log n_{i}}\le \fraction {C_2} {w (1 -1/n_k)}, $$
then $u \in\bl_w $ almost surely.
\end{cor}
Similarly, Corollary \ref {induction} gives:
\begin {cor}
\label {induction b}
Let $\x=\{\x_{ml}\} $  be a subnormal sequence, let $A>1 $,   $n_0 =1 $ 
and define $n_k $ by induction as 
$n_{k+1}=\min \{l \in \N: w(1 -1/l)A \le w (1 -1/n_k)
\}.$
If
$$\left(\sum_{m=n_{k-1}+1}^{n_{k}} m^2 \fa_m^2\right)^{1/2} \le \fraction {C} {w(1 -1/n_{k})\sqrt{\log n_{k}}}, $$
then $u(x,\x) =\sum_{m=0}^\infty \sum_{l=0}^{L_m} a_{ml} \x_{ml} r^{m} Y_{ml} \left(\fraction {x} {r}\right)  
\in\bl_w $ almost surely.
\end {cor}
The same results hold for analytic Bloch-type spaces as well, see the next section.

We will give examples of what the last corollary means for $w (r)=( 1 -r )^{\alpha }$ and $
 (\log \fraction1 { 1 -r} )^{-\alpha}  $ for $\alpha>0 $. The sequence $n_k $ can be chosen as $n_k=2^k$ and $n_k=2^{2^k} $, respectively, and a sufficient condition to be in $\bl_w $ almost surely when $w (r)=( 1 -r )^{\alpha }$ is 
$$\left(\sum_{m=2^{k-1}+1}^{2^{k}} m^2 \fa_m^2\right)^{1/2} \le C\fraction { 2^{\alpha k}} {\sqrt{k}}, $$
and for $w (r)= 
 (\log \fraction1 { 1 -r} )^{-\alpha}  $ 
it is
$$\left(\sum_{m=2^{2^{k-1}}+1}^{2^{2^{k}}} m^2 \fa_m^2\right)^{1/2} \le C  2^{\alpha 2^{k}-k/2}. 
 $$

In the same way as in Proposition \ref{sharpness} and \ref {constant} it can be shown that Corollary \ref {induction b} is sharp, just replace $a_j $ by $ja_j $ when defining the coefficients.

Proposition \ref{liminf} and \ref {fraction} can also be applied to Bloch-type functions in the disk: 
\begin {prop}
\label {liminf b}
Let $u (re^{i\ta}) = \sum_{j=0}^\infty (a_{j0} r^j \cos j\ta +a_{j1} r^j \sin j\ta) 
\in\bl_w $ and define 
a sequence $\{n_k\} $ as before. Let $k=k (j) $  be such that $n_{k -1} <j\le n_{k} $. 
Then
\beqs
\liminf_{j\rightarrow \infinity}| \fa_j |j w (1 -1/j) \sqrt{n_{k}}<\infty.  
\eeqs
Moreover, there exists a function  in $\bl_w $ for which $\liminf_{j\rightarrow \infinity} | \fa_j |\sqrt{n_{k}}/ {g (j)}>0 $, 
so the result  is sharp. 
\end {prop}
For  $w(r) = (1 -r)^\alpha $ this is $\liminf_{j\rightarrow \infinity}| \fa_j |j^{1 -\alpha}\sqrt{2^k}<\infty $, and since $n_k $ in this case does not grow very fast, this is equivalent to
 $$\liminf_{j\rightarrow \infinity}| \fa_j |j^{1 -\alpha}\sqrt{j}<\infty.  $$
For the usual Bloch functions we have $\liminf_{j\rightarrow \infinity}| \fa_j |\sqrt{j}<\infty.  $

\begin {prop}
\label {fraction b}
Assume $u (re^{i\ta}) = \sum_{j=0}^\infty (a_{j0} r^j \cos j\ta +a_{j1} r^j \sin j\ta) 
\in \bl_w$
and let $p_j $ be an increasing sequence of positive numbers such that $\lim_{j\rightarrow \infinity} p_j =\infinity $.
 Define $N(n) $ as the number of $\fa_j $ satisfying $j\le n $ and  $|\fa_j|\le \fraction{p_j } {jw (1 -1/j)\sqrt{j}}$. Then
$$\lim_{n\rightarrow \infinity} N (n)/n =1.$$
\end {prop}
This generalizes Corollary 2 in \cite {AK}, which is proved for Bloch functions.

\subsection {Analytic growth spaces and Bloch-type spaces}
Let $A_v^\infinity $ denote the space of analytic functions on $\D$ which fulfill $|u(z)|\le Kv(|z|)$ for some $K $, as mentioned in the introduction. We can prove a result similar to Theorem \ref {improved d} in this case as well, and this generalizes Theorem B.
The proof is similar to the proof of Theorem \ref {improved d}; we apply Theorem~C with F equal to the set of complex trigonometric polynomials.

\begin {thrm}
\label {improved analytic}
Let  $\x=\{\x_{m}\}$  be a subnormal sequence. If there exists an increasing sequence $\{n_k\} $ of positive  integers such that  for all $k $ we have
$g (n_{k+1})\le C_1 g (n_k) $  and
$$\sum_{j=1}^k \sqrt{\left(\sum_{m=n_{j-1}+1}^{n_j}|a_m |^2 
\right)
\log n_{j}}\le C_2 g (n_{k}), $$
then $u (z,\x) = \sum_{m=0}^\infty a_{m} \x_{m} z^m  \in A_v^\infinity$ almost surely.
\end {thrm}
A result similar to Corollary \ref {induction} follows easily. We can also apply Theorem \ref {improved analytic} to get  results similar to Corollary \ref {random improved} and Corollary \ref {induction b} for analytic Bloch-type spaces.


\section*{Acknowledgements}
I wish to thank my Ph.D. advisor Eugenia Malinnikova for suggesting this problem and
for valuable discussions. I also thank the referee for useful suggestions.

\end{document}